\documentclass[a4paper,12pt]{article}%
\usepackage{amsmath}
\usepackage{amsfonts}
\usepackage{amssymb}
\usepackage{graphicx}
\usepackage{hyperref}
\providecommand{\U}[1]{\protect\rule{.1in}{.1in}}
\newtheorem{theorem}{Theorem}[section]

\newtheorem{corollary}[theorem]{Corollary}

\newtheorem{lemma}[theorem]{Lemma}

\newtheorem{proposition}[theorem]{Proposition}
\newtheorem{remark}[theorem]{Remark}

\newenvironment{proof}[1][Proof]{\noindent\textbf{#1.} }{\ \rule{0.5em}{0.5em}}
\begin{document}

\title{An optimal pointwise Morrey-Sobolev inequality}
\author{Grey Ercole$^{\text{\thinspace a}}$\ and Gilberto A.
Pereira$^{\text{\thinspace b}}\medskip$\\{\small {$^{\mathrm{a}}$} Universidade Federal de Minas Gerais, Belo
Horizonte, MG, 30.123-970, Brazil}\\{\small grey@mat.ufmg.br}\\{\small {$^{\mathrm{b}}$} Universidade Federal de Ouro Preto, Ouro Preto, MG,
35.400-000, Brazil.}\\{\small gilberto.pereira@ufop.edu.br }}
\maketitle

\begin{abstract}
Let $\Omega$ be a bounded, smooth domain of $\mathbb{R}^{N},$ $N\geq1.$ For
each $p>N$ we study the optimal function $s=s_{p}$ in the pointwise inequality%
\[
\left\vert v(x)\right\vert \leq s(x)\left\Vert \nabla v\right\Vert
_{L^{p}(\Omega)},\quad\forall\,(x,v)\in\overline{\Omega}\times W_{0}%
^{1,p}(\Omega).
\]
We show that $s_{p}\in C_{0}^{0,1-(N/p)}(\overline{\Omega})$ and that $s_{p}$
converges pointwise to the distance function to the boundary, as
$p\rightarrow\infty.$ Moreover, we prove that if $\Omega$ is convex, then
$s_{p}$ is concave and has a unique maximum point.

\end{abstract}

{\small \noindent\textbf{2010 AMS Classification:} 35D40; 35J70; 35P30. }


\section{Introduction}

The well-known Morrey's inequality in $\mathbb{R}^{N}$ states that if $p>N$
then
\begin{equation}
\left\vert v(x)-v(y)\right\vert \leq C_{p,N}\left\vert x-y\right\vert
^{1-(N/p)}\left(  \int_{\mathbb{R}^{N}}\left\vert \nabla v\right\vert
^{p}\mathrm{d}x\right)  ^{\frac{1}{p}},\quad\forall\,x,y\in\mathbb{R}^{N}%
\quad\mathrm{and}\quad v\in W^{1,p}(\mathbb{R}^{N}), \label{MorreyRN}%
\end{equation}
where $C_{p,N}$ is a positive constant depending only on $p$ and $N,$ whose
optimal value is still unknown for $N\geq2.$

For $N=1$ the optimal constant $C_{p,1}$ in (\ref{MorreyRN}) is known to be
$1$ and, for $N\geq2,$ expressions that appear in standard proofs of
(\ref{MorreyRN}) are
\begin{equation}
C_{p,N}=\frac{2pN}{p-N}\quad\mathrm{and}\quad C_{p,N}=\frac{C(N)}%
{\sqrt[p]{N\omega_{N}}}\left(  \frac{p-1}{p-N}\right)  ^{\frac{p-1}{p}},
\label{CpN}%
\end{equation}
where $C(N)$ is a constant depending only on $N$ and $\omega_{N}$ is the
$N$-dimensional volume of the unit ball.

Now, let $\Omega$ be a bounded, smooth domain of $\mathbb{R}^{N}$ and let
$d_{\Omega}$ denote the distance function to the boundary $\partial\Omega,$
that is,
\[
d_{\Omega}(x):=\inf_{y\in\partial\Omega}\left\vert x-y\right\vert ,\quad
x\in\overline{\Omega}.
\]
Taking an arbitrary $y\in\partial\Omega$ in (\ref{MorreyRN}) one arrives at
the following pointwise inequality%
\begin{equation}
\left\vert v(x)\right\vert \leq C_{p,N}\left(  d_{\Omega}(x)\right)
^{1-(N/p)}\left\Vert \nabla v\right\Vert _{p},\quad\forall\,(x,v)\in
\overline{\Omega}\times W_{0}^{1,p}(\Omega), \label{sx}%
\end{equation}
where, for $N<p\leq\infty$, $\left\Vert \cdot\right\Vert _{p}$ stands for the
standard norm of $L^{p}(\Omega)$ (a notation that will be kept throughout the paper).

Note that $\left(  d_{\Omega}\right)  ^{1-(N/p)}\in C_{0}^{0,1-(N/p)}%
(\overline{\Omega}),$ the space of the functions that vanish on the boundary
$\partial\Omega$ and are $(1-(N/p))$-H\"{o}lder continuous in $\overline
{\Omega}.$

Passing to the maximum values in (\ref{sx}) we arrive at the well-known
Morrey-Sobolev inequality%
\begin{equation}
\left\Vert v\right\Vert _{\infty}\leq C_{p,N,\Omega}\left\Vert \nabla
v\right\Vert _{p},\quad\forall\,v\in W_{0}^{1,p}(\Omega), \label{MS}%
\end{equation}
where the constant $C_{p,N,\Omega}$ depends only on $p,N$ and $\Omega.$

In this paper we study the function%
\begin{equation}
s_{p}(x):=\left\{
\begin{array}
[c]{lll}%
\sup\left\{  \left\vert v(x)\right\vert /\left\Vert \nabla v\right\Vert
_{p}:v\in W_{0}^{1,p}(\Omega)\setminus\left\{  0\right\}  \right\}  &
\mathrm{if} & x\in\Omega\\
0 & \mathrm{if} & x\in\partial\Omega,
\end{array}
\right.  \label{spx}%
\end{equation}
which is the optimal function in the pointwise (version of) Morrey-Sobolev
inequality
\begin{equation}
\left\vert v(x)\right\vert \leq s(x)\left\Vert \nabla v\right\Vert _{p}%
,\quad\forall\,(x,v)\in\overline{\Omega}\times W_{0}^{1,p}(\Omega).
\label{sMS}%
\end{equation}

Clearly, $s_{p}$ satisfies (\ref{sMS}) and if $s:\overline{\Omega}%
\rightarrow\lbrack0,\infty)$ satisfies (\ref{sMS}), then $s_{p}\leq s$
pointwise in $\overline{\Omega}.$ This fact and (\ref{sx}) imply that
\begin{equation}
0<s_{p}(x)\leq C_{p,N}\left(  d_{\Omega}(x)\right)  ^{1-(N/p)}\quad
\forall\,x\in\Omega, \label{sharpness}%
\end{equation}
for every constant $C_{p,N}$ satisfying (\ref{MorreyRN}). Therefore, $s_{p}$
is continuous at the boundary points.

In Section \ref{Sec2} (see Theorem \ref{Main1}) we show that for each
$x\in\Omega$ given, there exists a (unique) function $u_{p}\in W_{0}%
^{1,p}(\Omega)$ that is positive in $\Omega$, assumes the maximum value $1$
uniquely at $x$ and satisfies%

\begin{equation}
s_{p}(x)=(\left\Vert \nabla u_{p}\right\Vert _{p})^{-1}. \label{2}%
\end{equation}
Using these facts and (\ref{sharpness}) we prove that $s_{p}\in C_{0}%
^{0,1-(N/p)}(\overline{\Omega}).$

We emphasize that, actually (see Remark \ref{Green}),
\[
s_{p}(x)=(G_{p}(x;x))^{\frac{p-1}{p}}\quad\mathrm{and}\quad u_{p}%
(y)=\frac{G_{p}(y;x)}{G_{p}(x;x)},\quad\forall\,y\in\Omega,
\]
where $G_{p}(\cdot;x)$ denotes the Green function of the $p$-Laplacian in
$\Omega$ with pole at $x.$

Alternatively, as it can be noticed from \cite{Jan},
\[
s_{p}(x)=(\operatorname{cap}_{p}(\left\{  x\right\}  ,\Omega))^{-\frac{1}{p}}%
\]
for each $x\in\Omega,$ where%
\[
\operatorname{cap}_{p}(\left\{  x\right\}  ,\Omega):=\inf\left\{  \left\Vert
\nabla u\right\Vert _{p}^{p}:u\in W_{0}^{1,p}(\Omega)\cap C(\Omega),\quad
u(x)\geq1\right\}
\]
denotes the $p$-capacity of the punctured domain $\Omega\setminus\left\{
x\right\}  .$ Hence, as consequence of (\ref{2}), $u_{p}$ is the
$p$-capacitary function corresponding to $\operatorname{cap}_{p}(\left\{
x\right\}  ,\Omega)$.

Still in Section \ref{Sec2} (see Corollary \ref{1d}), we derive an explicit
expression of $s_{p}$ for the unidimensional case, where $\Omega$ is an
interval. We also argue that in the case where $\Omega$ is a multidimensional
ball the function $s_{p}$ is radially symmetric and radially decreasing. Even
though, it seems to be very difficult to derive an explicit expression for
$s_{p}$ in this case. We recall that an explicit expression for the Green
function of the $p$-Laplacian for a ball is not available if $p>2.$

In Section \ref{Sec3} (see Proposition \ref{uinf}) we prove that
\[
\lim_{p\rightarrow\infty}s_{p}(x)=d_{\Omega}(x),\quad\forall\,x\in
\overline{\Omega}.
\]
Moreover, for each $x\in\Omega$ we show that the function $u_{p}$ satisfying
(\ref{2}) converges uniformly, as $p\rightarrow\infty$, to a function
$u_{\infty}\in W^{1,\infty}(\Omega)\cap C_{0}(\overline{\Omega})$ that is
infinity harmonic in the punctured domain $\Omega\setminus\left\{  x\right\}
$ (see Theorem \ref{Main2}).

In Section \ref{Sec4} (see Theorem \ref{convex1}), we prove that if $\Omega$
is convex, then the function $s_{p}$ is concave and has a unique maximum
point. The concavity proof is adapted from arguments developed by Hynd and
Lindgren \cite{HL}. The uniqueness of the maximum point is a direct
consequence of their main result: the extremal functions for the
Morrey-Sobolev inequality (\ref{MS}) are scalar multiple of each other and
achieve the maximum value uniquely at a same point.

\section{The optimal function$\label{Sec2}$}

In this section, $\Omega$ is a bounded, smooth domain of $\mathbb{R}^{N}$ and
$p>N\geq1.$ For each $x\in\Omega$ we define
\[
\mathcal{M}_{p}(x):=\left\{  v\in W_{0}^{1,p}(\Omega):\left\vert
v(x)\right\vert =\left\Vert v\right\Vert _{\infty}=1\right\}
\]
and
\begin{equation}
\mu_{p}(x):=\min_{v\in\mathcal{M}_{p}(x)}\left\Vert \nabla v\right\Vert
_{p}^{p}. \label{lampS1}%
\end{equation}

We recall that the Dirac Delta distribution $\delta_{x}$ is the linear
functional defined by%
\[
\left\langle \delta_{x},\phi\right\rangle :=\phi(x),\quad\forall\,\phi\in
W_{0}^{1,p}(\Omega).
\]
Actually, by virtue of (\ref{MS}), $\delta_{x}$ belongs to the dual of
$W_{0}^{1,p}(\Omega),$ commonly denoted by $W^{-1,p^{\prime}}(\Omega),$
$(1/p)+(1/p^{\prime})=1.$

\begin{proposition}
\label{mumin}Let $x\in\Omega$ be fixed. There exists $v\in\mathcal{M}_{p}(x)$
such that
\begin{equation}
\mu_{p}(x)=\left\Vert \nabla v\right\Vert _{p}^{p}. \label{aux4}%
\end{equation}

\end{proposition}

\begin{proof}
Let $\left\{  v_{n}\right\}  \subset\mathcal{M}_{p}(x)$ be such that
$\left\Vert \nabla v_{n}\right\Vert _{p}^{p}\rightarrow\mu_{p}(x).$ As
$W_{0}^{1,p}(\Omega)$ is reflexive and compactly embedded in $C(\overline
{\Omega})$ we can assume, without loss of generality, that%
\[
v_{n}\rightharpoonup v\quad\mathrm{weakly\,in\,}W_{0}^{1,p}(\Omega
)\mathrm{\quad and\quad}v_{n}\rightarrow v\quad
\mathrm{uniformly\mathrm{\,in\,}}C(\overline{\Omega}),
\]
for some $v\in W_{0}^{1,p}(\Omega)\cap C(\overline{\Omega}).$ As $\left\Vert
v_{n}\right\Vert _{\infty}=\left\vert v_{n}(x)\right\vert =1,$ the uniform
convergence implies that $v\in\mathcal{M}_{p}(x).$ Hence, (\ref{aux4}) follows
since the weak convergence yields
\[
\mu_{p}(x)\leq\left\Vert \nabla v\right\Vert _{p}^{p}\leq\liminf
_{n\rightarrow\infty}\left\Vert \nabla v_{n}\right\Vert _{p}^{p}=\mu_{p}(x).
\]

\end{proof}

\begin{proposition}
\label{prop0}Let $x\in\Omega$ be fixed and let $u_{p}\in W_{0}^{1,p}(\Omega)$
be the only weak solution of the Dirichlet problem
\begin{equation}
\left\{
\begin{array}
[c]{rrll}%
-\Delta_{p}u & = & \mu_{p}(x)\delta_{x} & \mathrm{in}\,\Omega\\
u & = & 0 & \mathrm{on}\,\partial\Omega.
\end{array}
\right.  \label{updir}%
\end{equation}
Then, $u_{p}\in\mathcal{M}_{p}(x),$ is strictly positive in $\Omega,$ attains
its maximum value only at $x$ and
\[
\mu_{p}(x)=\left\Vert \nabla u_{p}\right\Vert _{p}^{p}.
\]

\end{proposition}

\begin{proof}
The existence and the uniqueness of $u_{p}$ follow from the bijectivity of the
duality mapping (see \cite{DJM}) from $W_{0}^{1,p}(\Omega)$ into
$W_{0}^{-1,p^{\prime}}(\Omega)$ given by%
\[%
\begin{array}
[c]{cll}%
u & \mapsto & \left\langle -\Delta_{p}u,\phi\right\rangle :=\int_{\Omega
}\left\vert \nabla u\right\vert ^{p-2}\nabla u\cdot\nabla\phi\mathrm{d}x.
\end{array}
\]
Thus, since $\mu_{p}(x)\delta_{x}\in W^{-1,p^{\prime}}(\Omega)$, there exists
a unique function $u_{p}\in W_{0}^{1,p}(\Omega)$ satisfying%
\begin{equation}
\int_{\Omega}\left\vert \nabla u_{p}\right\vert ^{p-2}\nabla u_{p}\cdot
\nabla\phi\mathrm{d}y=\mu_{p}(x)\phi(x),\quad\forall\,\phi\in W_{0}%
^{1,p}(\Omega),\label{weakup}%
\end{equation}
which means that $u_{p}$ is the only weak solution of (\ref{updir}).

Taking an arbitrary nonnegative test function $\phi$ in (\ref{weakup}) we
conclude, by the weak comparison principle, that $u_{p}\geq0$ in $\Omega.$
Since $\mu_{p}(x)>0,$ the identity (\ref{weakup}) also guarantees that $u_{p}$
is not the null function. Using $\phi=u_{p}$ in (\ref{weakup}) we obtain%
\begin{equation}
\mu_{p}(x)u_{p}(x)=\left\Vert \nabla u_{p}\right\Vert _{p}^{p}. \label{prop0a}%
\end{equation}

Moreover, considering in (\ref{weakup}) an arbitrary test function $\phi$
supported in the punctured domain $\Omega\setminus\left\{  x\right\}  $ we can
see that $u_{p}$ is $p$-harmonic in this domain (i.e. $\Delta_{p}u_{p}=0$ in
$\Omega\setminus\left\{  x\right\}  $ in the weak sense). It follows that the
minimum and maximum values of $u_{p}$ are necessarily attained on the boundary
$\partial\Omega\cup\left\{  x\right\}  $ of $(\Omega\setminus\left\{
x\right\}  )$ (see \cite{Lq}). Consequently (recalling that $u_{p}=0$ on
$\partial\Omega$),
\[
0<u_{p}(y)<u_{p}(x)=\left\Vert u_{p}\right\Vert _{\infty}\quad\forall
\,y\in\Omega\setminus\left\{  x\right\}  .
\]

Combining (\ref{prop0a}) with the definition of $\mu_{p}(x)$ in (\ref{lampS1})
and observing that $u_{p}/\left\Vert u_{p}\right\Vert _{\infty}\in
\mathcal{M}_{p}(x)$ we arrive at%
\[
\frac{\left\Vert \nabla u_{p}\right\Vert _{p}^{p}}{u_{p}(x)}=\mu_{p}%
(x)\leq\frac{\left\Vert \nabla u_{p}\right\Vert _{p}^{p}}{\left\Vert
u_{p}\right\Vert _{\infty}^{p}},
\]
from which follows that $\left\Vert u_{p}\right\Vert _{\infty}\leq1.$

Now, let $v\in\mathcal{M}_{p}(x)$ such that $\mu_{p}(x)=\left\Vert \nabla
v\right\Vert _{p}^{p}$ (the existence of $v$ comes from the previous
proposition). As $\left\vert v\right\vert \in\mathcal{M}_{p}(x)$ and
$\left\Vert \nabla\left\vert v\right\vert \right\Vert _{p}=\left\Vert \nabla
v\right\Vert _{p}=\mu_{p}(x)$ we can take $\phi=\left\vert v\right\vert $ in
(\ref{weakup}) and use H\"{o}lder inequality to find
\begin{equation}
\mu_{p}(x)=\int_{\Omega}\left\vert \nabla u_{p}\right\vert ^{p-2}\nabla
u_{p}\cdot\nabla\left\vert v\right\vert \mathrm{d}y\leq\left\Vert \nabla
u_{p}\right\Vert _{p}^{p-1}\left\Vert \nabla\left\vert v\right\vert
\right\Vert _{p}=\left\Vert \nabla u_{p}\right\Vert _{p}^{p-1}(\mu
_{p}(x))^{1/p}. \label{prop0c}%
\end{equation}
Consequently,
\[
\mu_{p}(x)\leq\left\Vert \nabla u_{p}\right\Vert _{p}^{p},
\]
an inequality that, in view of (\ref{prop0a}), implies that $u_{p}(x)\geq1.$
It follows that $\left\Vert u_{p}\right\Vert _{\infty}=1$ (recall that
$\left\Vert u_{p}\right\Vert _{\infty}=u_{p}(x)$ and $\left\Vert
u_{p}\right\Vert _{\infty}\leq1$). This shows that $u_{p}\in\mathcal{M}%
_{p}(x)$ and, in view of (\ref{prop0a}), yields
\[
\mu_{p}(x)=\left\Vert \nabla u_{p}\right\Vert _{p}^{p}.
\]
Hence, H\"{o}lder's inequality in (\ref{prop0c}) becomes an equality and this
implies that $u_{p}=\left\vert v\right\vert .$ As $u_{p}>0$ in $\Omega,$ we
conclude that $v$ does not change sign in $\Omega,$ so that either $v=u_{p}$
or $v=-u_{p}.$
\end{proof}

In the sequel, $s_{p}$ denotes the best function in the pointwise inequality
(\ref{sMS}), defined by (\ref{spx}).

\begin{corollary}
One has
\[
s_{p}(x)=(\mu_{p}(x))^{-1/p},\quad\forall\,x\in\Omega.
\]

\end{corollary}

\begin{proof}
Let $u_{p}\in\mathcal{M}_{p}(x)$ given by Proposition \ref{prop0} and take and
arbitrary $v\in W_{0}^{1,p}(\Omega)\setminus\left\{  0\right\}  .$ We have, by
H\"{o}lder inequality,%
\begin{align*}
\mu_{p}(x)\left\vert v(x)\right\vert  &  =\left\vert \mu_{p}(x)v(x)\right\vert
\\
&  =\left\vert \int_{\Omega}\left\vert \nabla u_{p}\right\vert ^{p-2}\nabla
u_{p}\cdot\nabla v\mathrm{d}y\right\vert \leq\left\Vert \nabla u_{p}%
\right\Vert _{p}^{p-1}\left\Vert \nabla v\right\Vert _{p}=(\mu_{p}%
(x))^{1-\frac{1}{p}}\left\Vert \nabla v\right\Vert _{p}.
\end{align*}
It follows that $(\mu_{p}(x))^{-1/p}\geq\left\vert v(x)\right\vert /\left\Vert
\nabla v\right\Vert _{p}.$ The arbitrariness of $v$ and (\ref{spx}) imply that
$(\mu_{p}(x))^{-1/p}\geq s_{p}(x).$

Recalling that $(\mu_{p}(x))^{-1/p}=\left\vert u_{p}(x)\right\vert /\left\Vert
\nabla u_{p}\right\Vert _{p}$ and that $\left\vert u_{p}(x)\right\vert
/\left\Vert \nabla u_{p}\right\Vert _{p}\leq s_{p}(x)$ we conclude that
$s_{p}(x)=(\mu_{p}(x))^{-1/p}.$
\end{proof}

\begin{corollary}
\label{Regularity}One has
\begin{equation}
\left\vert s_{p}(x)-s_{p}(y)\right\vert \leq C_{p,N}\left\vert x-y\right\vert
^{1-(N/p)},\quad\forall\,x,y\in\overline{\Omega}, \label{spHolder}%
\end{equation}
for every constant $C_{p,N}$ satisfying (\ref{MorreyRN}). Consequently,
$s_{p}\in C_{0}^{0,1-(N/p)}(\overline{\Omega}).$
\end{corollary}

\begin{proof}
Obviously, (\ref{spHolder}) implies that $s_{p}\in C_{0}^{0,1-(N/p)}%
(\overline{\Omega}).$ So, let us prove (\ref{spHolder}).

Let $x,y\in\overline{\Omega}.$ If $y\in\partial\Omega,$ then (\ref{sharpness})
yields%
\[
\left\vert s_{p}(x)-s_{p}(y)\right\vert =s_{p}(x)\leq C_{p,N}\left(
d_{\Omega}(x)\right)  ^{1-(N/p)}\leq C_{p,N}\left\vert x-y\right\vert
^{1-(N/p)}.
\]
Likewise, (\ref{spHolder}) holds if $x\in\partial\Omega.$

Now, we assume that $x,y\in\Omega.$ Let $u\in W_{0}^{1,p}(\Omega)$ be a
positive function such that
\[
u(x)=s_{p}(x)\left\Vert \nabla u\right\Vert _{p}%
\]
(take $u$ a positive multiple of the function $u_{p}\in\mathcal{M}_{p}(x)$
given by Proposition \ref{prop0}). As
\[
u(y)\leq s_{p}(y)\left\Vert \nabla u\right\Vert _{p}%
\]
we have, in view of (\ref{MorreyRN}),%
\[
(s_{p}(x)-s_{p}(y))\left\Vert \nabla u\right\Vert _{p}\leq u(x)-u(y)\leq
\left\vert u(x)-u(y)\right\vert \leq C_{p,N}\left\vert x-y\right\vert
^{1-(N/p)}\left\Vert \nabla u\right\Vert _{p}.
\]
As $\left\Vert \nabla u\right\Vert _{p}>0$ we get%
\[
s_{p}(x)-s_{p}(y)\leq C_{p,N}\left\vert x-y\right\vert ^{1-(N/p)}.
\]

Analogously, by taking a function $v\in W_{0}^{1,p}(\Omega)$ such that
$v(y)=s_{p}(y)\left\Vert \nabla v\right\Vert _{p}$ we arrive at the inequality%
\[
s_{p}(y)-s_{p}(x)\leq C_{p,N}\left\vert x-y\right\vert ^{1-(N/p)},
\]
completing thus the proof.
\end{proof}

We summarize the main results above in the following theorem.

\begin{theorem}
\label{Main1}Let $x\in\Omega$ be fixed and let $u_{p}\in W_{0}^{1,p}(\Omega)$
be the only weak solution of
\begin{equation}
\left\{
\begin{array}
[c]{rrll}%
-\Delta_{p}u & = & (s_{p}(x))^{-p}\delta_{x} & \mathrm{in}\,\Omega\\
u & = & 0 & \mathrm{on}\,\partial\Omega.
\end{array}
\right.  \label{edpsp}%
\end{equation}
Then, $u_{p}\in\mathcal{M}_{p}(x),$ is strictly positive in $\Omega,$ attains
its maximum value only at $x,$ and%
\[
s_{p}(x)=(\left\Vert \nabla u_{p}\right\Vert _{p})^{-1}.
\]
Moreover,

\begin{enumerate}
\item[(a)] $\left\vert v(x)\right\vert \leq s_{p}(x)\left\Vert \nabla
v\right\Vert _{p},\quad\forall\,(x,v)\in\overline{\Omega}\times W_{0}%
^{1,p}(\Omega);$

\item[(b)] $\left\vert v(x)\right\vert =s_{p}(x)\left\Vert \nabla v\right\Vert
_{p}$ if, and only if, $v$ is a scalar multiple of $u_{p};$

\item[(c)] $s_{p}\in C_{0}^{0,1-(N/p)}(\overline{\Omega}).$
\end{enumerate}
\end{theorem}

\begin{proposition}
\label{corol}Let $x\in\Omega,$ $u\in W_{0}^{1,p}(\Omega)$ and $\mu>0$ be such
that
\[
\left\{
\begin{array}
[c]{rrll}%
-\Delta_{p}u & = & \mu\delta_{x} & \mathrm{in}\,\Omega\\
u & = & 0 & \mathrm{on}\,\partial\Omega.
\end{array}
\right.
\]
If $u(x)=1$ then,
\[
u=u_{p}\quad\mathrm{and}\quad\mu=(s_{p}(x))^{-p},
\]
where $u_{p}$ denotes the only solution of (\ref{edpsp}).
\end{proposition}

\begin{proof}
Since $\left\Vert \nabla u\right\Vert _{p}^{p}=\mu$ and $s_{p}(x)=(\left\Vert
\nabla u_{p}\right\Vert _{p})^{-1},$ H\"{o}lder's inequality yields
\[
\mu=\mu u_{p}(x)=\int_{\Omega}\left\vert \nabla u\right\vert ^{p-2}\nabla
u\cdot\nabla u_{p}\mathrm{d}y\leq\left\Vert \nabla u\right\Vert _{p}%
^{p-1}\left\Vert \nabla u_{p}\right\Vert _{p}=\mu^{\frac{p-1}{p}}%
(s_{p}(x))^{-1},
\]
so that $~\mu^{1/p}\leq(s_{p}(x))^{-1}.$ Using this and recalling that
$u(x)=1$ we have%
\[
1=u(x)\leq s_{p}(x)\left\Vert \nabla u\right\Vert _{p}=s_{p}(x)\mu^{1/p}%
\leq1.
\]

Therefore, $\mu=(s_{p}(x))^{-p}$ and, by uniqueness, $u=u_{p}.$
\end{proof}

\begin{remark}
\label{Green}For each $x\in\Omega$ let $G_{p}(\cdot;x)$ denote the Green
function of the $p$-Laplacian in $\Omega$ with pole at $x.$ That is,
$G_{p}(\cdot;x)$ is the (only) solution of
\[
\left\{
\begin{array}
[c]{rrll}%
-\Delta_{p}u & = & \delta_{x} & \mathrm{in}\,\Omega\\
u & = & 0 & \mathrm{on}\,\partial\Omega,
\end{array}
\right.
\]
where $\delta_{x}$ denotes the Dirac delta distribution supported at $x.$
Since $G_{p}(x;x)/G_{p}(x;x)=1$ and
\[
-\Delta_{p}(G_{p}(\cdot;x)/G_{p}(x;x))=G_{p}(x;x))^{1-p}\delta_{x}%
\]
an immediate consequence of Proposition \ref{corol} is that
\[
s_{p}(x)=(G_{p}(x;x))^{\frac{p-1}{p}}\quad\mathrm{and}\quad u_{p}%
(y)=\frac{G_{p}(y;x)}{G_{p}(x;x)},\quad\forall\,y\in\Omega.
\]

\end{remark}

In the unidimensional case, $s_{p}$ and $u_{p}$ are given by explicit
expressions, as the following corollary shows.

\begin{corollary}
\label{1d}Let $p>N=1$ and $\Omega=(a,b).$ For each $x\in(a,b)$ one has
\begin{equation}
s_{p}(x)=\left(  (x-a)^{1-p}+(b-x)^{1-p}\right)  ^{-1/p} \label{spunid}%
\end{equation}
and%
\begin{equation}
u_{p}(y):=\left\{
\begin{array}
[c]{ccc}%
(y-a)(x-a)^{-1} & \mathrm{if} & a\leq y\leq x\\
(b-y)(b-x)^{-1} & \mathrm{if} & x\leq y\leq b.
\end{array}
\right.  \label{un=1}%
\end{equation}

\end{corollary}

\begin{proof}
Let $\mu$ be the right-hand side of (\ref{spunid}) raised to $-p,$ that is,
\[
\mu=\left(  (x-a)^{1-p}+(b-x)^{1-p}\right)  .
\]
Let $u$ be expressed by the right-hand side of (\ref{un=1}). Clearly, $\mu>0$
and $u(x)=1.$

For $\phi\in W_{0}^{1,p}((a,b))$ given, we have%
\begin{align*}
\int_{a}^{b}\left\vert u^{\prime}\right\vert ^{p-2}u^{\prime}\phi^{\prime
}\mathrm{d}y  &  =\int_{a}^{x}(x-a)^{1-p}\phi^{\prime}\mathrm{d}y-\int_{x}%
^{b}(b-x)^{1-p}\phi^{\prime}\mathrm{d}y\\
&  =(x-a)^{1-p}\int_{a}^{x}\phi^{\prime}\mathrm{d}y-(b-x)^{1-p}\int_{x}%
^{b}\phi^{\prime}\mathrm{d}y\\
&  =\left(  (x-a)^{1-p}+(b-x)^{1-p}\right)  \phi(x)=\mu\phi(x).
\end{align*}
Thus, according to Proposition \ref{corol}, $\mu=(s_{p}(x))^{-p}$ and
$u=u_{p}.$
\end{proof}

(Note that $s_{p}$ is symmetric with respect to $\overline{x}:=(a+b)/2$ . It
is also simple to check that $s_{p}$ is concave.)

We end this section with some remarks on the case where $\Omega=B_{R}(0),$ the
$N$-dimensional ball ($N\geq2$) centered at the origin with radius $R.$ In
this case, the function $x\mapsto s_{p}(x)$ is radially symmetric:
$s_{p}(x)=s_{p}(y)$ whenever $\left\vert x\right\vert =\left\vert y\right\vert
.$ Indeed, by using an orthogonal change of variable one can see that the only
positive maximizer of $s_{p}(x)$ in $\mathcal{M}_{p}(x)$ is a rotation of the
only positive maximizer of $s_{p}(y)$ in $\mathcal{M}_{p}(y).$ Note that the
function $u_{p}$ corresponding to $x\not =0$ cannot be radial since $x$ is its
unique maximum point.

On the other hand, $u_{p}$ is radial when $x=0.$ By the way,%
\[
s_{p}(0)=\frac{R^{1-(N/p)}}{\sqrt[p]{N\omega_{N}}}\left(  \frac{p-1}%
{p-N}\right)  ^{\frac{p-1}{p}}\quad\mathrm{and}\quad u_{p}(y)=1-\left(
\left\vert y\right\vert /R\right)  ^{\frac{p-N}{p-1}}\quad\forall\,y\in
B_{R}(0),
\]
since the best constant in the Morrey-Sobolev inequality (\ref{MS}) for
$B_{R}(0)$ is given by the first expression above and its corresponding
extremal functions are scalar multiples of the function given by the second
expression above (see \cite{EP, Lq}).

As consequence of Theorem \ref{convex1} (Section \ref{Sec4}), the function
$s_{p}$ is concave and assumes its maximum value uniquely at $0.$ Hence, we
conclude that $s_{p}$ is radially decreasing.

Even knowing these properties of $s_{p}$ it seems to be very difficult to
compute this function explicitly (at $x\not =0$). Note that the Green function
$G_{p}$ of the $p$-Laplacian for a ball is not known if $p>2.$

\section{Asymptotics as $p\rightarrow\infty$\label{Sec3}}

An immediate lower bound to the function $s_{p}$ comes from its definition
(\ref{spx}), by taking $v=d_{\Omega}$ (and recalling that $\left\vert \nabla
d_{\Omega}\right\vert =1$ in $\Omega$):%
\begin{equation}
d_{\Omega}(x)\left\vert \Omega\right\vert ^{-\frac{1}{p}}\leq s_{p}%
(x),\quad\forall\,x\in\overline{\Omega}. \label{lowsp}%
\end{equation}

\begin{lemma}
For each fixed $x\in\overline{\Omega},$ the function $p\mapsto s_{p}%
(x)\left\vert \Omega\right\vert ^{\frac{1}{p}}$ is nonincreasing and%
\begin{equation}
d_{\Omega}(x)\leq s(x):=\lim_{p\rightarrow\infty}s_{p}(x)=\inf_{q>N}%
s_{q}(x)\left\vert \Omega\right\vert ^{\frac{1}{q}}. \label{est+}%
\end{equation}

\end{lemma}

\begin{proof}
Let $N<p_{1}<p_{2}.$ For $i\in\left\{  1,2\right\}  $ let $u_{i}\in
\mathcal{M}_{p_{i}}(x)$ be such that%
\[
\mu_{p_{i}}(x)=\left\Vert \nabla u_{i}\right\Vert _{p_{i}}^{p_{i}}.
\]
Since $\mathcal{M}_{p_{2}}(x)\subset\mathcal{M}_{p_{1}}(x)$ we obtain, by
H\"{o}lder's inequality,
\[
\mu_{p_{1}}(x)\leq\left\Vert \nabla u_{2}\right\Vert _{p_{1}}^{p_{1}}%
\leq\left\Vert \nabla u_{2}\right\Vert _{p_{2}}^{p_{1}}\left\vert
\Omega\right\vert ^{1-\frac{p_{1}}{p_{2}}}=\left(  \mu_{p_{2}}(x)\right)
^{p_{1}/p_{2}}\left\vert \Omega\right\vert ^{1-\frac{p_{1}}{p_{2}}}.
\]
This means that
\[
s_{p_{2}}(x)\left\vert \Omega\right\vert ^{1/p_{2}}=\left(  \mu_{p_{2}%
}(x)\right)  ^{-1/p_{2}}\left\vert \Omega\right\vert ^{1/p_{2}}\leq\left(
\mu_{p_{1}}(x)\right)  ^{-1/p_{1}}\left\vert \Omega\right\vert ^{1/p_{1}%
}=s_{p_{1}}(x)\left\vert \Omega\right\vert ^{1/p_{1}}.
\]

It follows that
\[
\lim_{p\rightarrow\infty}s_{p}(x)\left\vert \Omega\right\vert ^{1/p}%
=\inf_{q>N}s_{q}(x)\left\vert \Omega\right\vert ^{1/q}.
\]
Hence, since $s_{p}(x)=(s_{p}(x)\left\vert \Omega\right\vert ^{1/p})\left\vert
\Omega\right\vert ^{-1/p},$ the limit $s(x)$ in (\ref{est+}) exists and
coincides with the above limit. The first inequality in (\ref{est+}) then
follows by letting $p\rightarrow\infty$ in (\ref{lowsp}).
\end{proof}

The next result shows that the inequality in (\ref{est+}) is, in fact, an equality.

\begin{proposition}
\label{uinf}Let $x\in\Omega$ be fixed and, for each $p>N,$ let $u_{p}%
\in\mathcal{M}_{p}(x)$ be the positive function such that $s_{p}%
(x)=(\left\Vert \nabla u_{p}\right\Vert _{p})^{-1}.$ We claim that every
sequence $\left\{  u_{p_{n}}\right\}  _{n\in\mathbb{N}},$ with $p_{n}%
\rightarrow\infty,$ admits a subsequence $\left\{  u_{p_{n_{j}}}\right\}
_{j\in\mathbb{N}}$ converging uniformly to a nonnegative function $u_{\infty
}\in W^{1,\infty}(\Omega)\cap C_{0}(\overline{\Omega})$ such that%
\[
u_{\infty}(x)=\left\Vert u_{\infty}\right\Vert _{\infty}=1\quad\mathrm{and}%
\quad\left\Vert \nabla u_{\infty}\right\Vert _{\infty}=(d_{\Omega}(x))^{-1}.
\]
Moreover,%
\begin{equation}
\lim_{p\rightarrow\infty}s_{p}(x)=d_{\Omega}(x),\quad\forall\,x\in
\overline{\Omega}. \label{aux3}%
\end{equation}

\end{proposition}

\begin{proof}
Let $\left\{  p_{n}\right\}  _{n\in\mathbb{N}}\subset(N,\infty)$ be such that
$p_{n}\rightarrow\infty$ and fix $r>N.$ There exists $n_{0}>N$ such that
$p_{n}>r$ for every $n>n_{0}.$ Hence, by H\"{o}lder's inequality and
(\ref{est+}),
\begin{equation}
\left\Vert \nabla u_{p_{n}}\right\Vert _{r}\leq\left\Vert \nabla u_{p_{n}%
}\right\Vert _{p_{n}}\left\vert \Omega\right\vert ^{\frac{1}{r}-\frac{1}%
{p_{n}}}=(s_{p_{n}}(x))^{-1}\left\vert \Omega\right\vert ^{-1/p_{n}}\left\vert
\Omega\right\vert ^{1/r},\quad\forall\,n>n_{0}. \label{aux2}%
\end{equation}
That is, $\left\{  u_{p_{n}}\right\}  _{n>n_{0}}$ is bounded in $W_{0}%
^{1,r}(\Omega).$

Therefore, we can assume (passing to a subsequence if necessary) that
$u_{p_{n}}$ converges to a nonnegative function $u_{\infty}\in W_{0}%
^{1,r}(\Omega)\cap C(\overline{\Omega})$ uniformly in $C(\overline{\Omega})$
and weakly in $W_{0}^{1,r}(\Omega).$ The uniform convergence implies that
$u_{\infty}(x)=\left\Vert u_{\infty}\right\Vert _{\infty}=1$ (recall that
$\left\Vert u_{p_{n}}\right\Vert _{\infty}=u_{p_{n}}(x)=1,$ since $u_{p_{n}%
}\in S_{p_{n}}^{1}(x)$) whereas the weak convergence and (\ref{aux2}) yield%
\begin{equation}
\left\Vert \nabla u_{\infty}\right\Vert _{r}\leq\liminf_{n\rightarrow\infty
}\left\Vert \nabla u_{p_{n}}\right\Vert _{r}\leq\left\vert \Omega\right\vert
^{1/r}(s(x))^{-1}. \label{aux1}%
\end{equation}

Moreover, using Morrey's inequality (\ref{MorreyRN}) with the first expression
in (\ref{CpN}), we obtain in sequence%
\begin{align*}
\left\vert u_{p_{n}}(y)-u_{p_{n}}(z)\right\vert  &  \leq\frac{2Np_{n}}%
{p_{n}-N}\left\Vert \nabla u_{p_{n}}\right\Vert _{p_{n}}\left\vert
y-z\right\vert ^{1-(N/p)}\\
&  =\frac{2Np_{n}}{p_{n}-N}(s_{p_{n}}(x))^{-1}\left\vert y-z\right\vert
^{1-(N/p)},\quad\forall\,y,z\in\overline{\Omega},
\end{align*}
and
\[
\left\vert u_{\infty}(y)-u_{\infty}(z)\right\vert \leq2N(s(x))^{-1}\left\vert
y-z\right\vert ,\quad\forall\,y,z\in\overline{\Omega}.
\]
It follows that $u_{\infty}\in W^{1,\infty}(\Omega)\cap C_{0}(\overline
{\Omega}),$ so that its Lipschitz constant is $\left\Vert \nabla u_{\infty
}\right\Vert _{\infty}.$

The arbitrariness of $r>N$ allows us to let $r\rightarrow\infty$ in
(\ref{aux1}) to conclude that
\begin{equation}
\left\Vert \nabla u_{\infty}\right\Vert _{\infty}=\lim_{r\rightarrow\infty
}\left\Vert \nabla u_{\infty}\right\Vert _{r}\leq\lim_{r\rightarrow\infty
}\left\vert \Omega\right\vert ^{1/r}(s(x))^{-1}=(s(x))^{-1}. \label{est1}%
\end{equation}

Now, picking $y\in\partial\Omega$ such that $d_{\Omega}(x)=\left\vert
x-y\right\vert ,$ we obtain from (\ref{est1})%
\[
1=u_{\infty}(x)=u_{\infty}(x)-u_{\infty}(y)\leq\left\Vert \nabla u_{\infty
}\right\Vert _{\infty}\left\vert x-y\right\vert =\left\Vert \nabla u_{\infty
}\right\Vert _{\infty}d_{\Omega}(x)\leq d_{\Omega}(x)(s(x))^{-1}\leq1,
\]
from which follows that $s(x)=d_{\Omega}(x)=(\left\Vert \nabla u_{\infty
}\right\Vert _{\infty})^{-1}.$
\end{proof}

\begin{remark}
It is known that $d_{\Omega}$ is concave whenever $\Omega$ is convex. This
fact can be proved directly, but it also follows from Theorem \ref{convex1}
(Section \ref{Sec4}) and (\ref{aux3}).
\end{remark}

Following step by step the proof of Theorem 3.11 of \cite{EP} we can show that
$u_{\infty}$ is a viscosity solution of
\begin{equation}
\left\{
\begin{array}
[c]{rrll}%
\Delta_{\infty}u & = & 0 & \mathrm{in}\,\Omega\setminus\left\{  x\right\} \\
u & = & d_{\Omega}/d_{\Omega}(x) & \mathrm{on}\,\partial\Omega\cup\left\{
x\right\}  ,
\end{array}
\right.  \label{infdiric}%
\end{equation}
where
\[
\Delta_{\infty}u:=\sum_{i,j=1}^{N}u_{x_{i}}u_{x_{j}}u_{x_{i}x_{j}}%
\]
is the infinity Laplacian operator. (We refer to \cite{Lq1} to the concept of
viscosity solution.)

It turns out that (\ref{infdiric}) has a unique viscosity solution $u\in
C(\overline{\Omega}).$ This uniqueness result follows from the comparison
principle for the $\infty$-harmonic equation in the domain $\Omega
\setminus\left\{  x\right\}  ,$ which can be quoted from \cite{BB, Jensen}.

Therefore, $u_{\infty}$ is the uniform limit of the family $\left\{
u_{p}\right\}  _{p>N},$ as $p\rightarrow\infty$ (which means: $u_{p_{n}%
}\rightarrow u_{\infty}$ uniformly in $\overline{\Omega},$ for any sequence
$\left\{  u_{p_{n}}\right\}  $ with $p_{n}\rightarrow\infty$). Actually, we
have the following theorem.

\begin{theorem}
\label{Main2}Let $x\in\Omega$ be fixed and, for each $p>N,$ let $u_{p}%
\in\mathcal{M}_{p}(x)$ be the positive function such that $s_{p}%
(x)=(\left\Vert \nabla u_{p}\right\Vert _{p})^{-1}.$ The function $u_{\infty
}\in C_{0}(\overline{\Omega})\cap W^{1,\infty}(\Omega)$ is the uniform limit
in $\overline{\Omega}$ of the family $\left\{  u_{p}\right\}  ,$ as
$p\rightarrow\infty.$ Moreover, $u_{\infty}$ is strictly positive in $\Omega,$
attains its maximum value $1$ uniquely at $x$ and is the only viscosity
solution of (\ref{infdiric}).
\end{theorem}

\begin{proof}
Since $u_{\infty}(x)=1>0$ and $u_{\infty}=0\ $on $\partial\Omega,$ the strict
positiveness of $u_{\infty}$ in $\Omega\setminus\left\{  x\right\}  $ follows
from the Harnack's inequality for the $\infty$-harmonic equation in balls
contained in $\Omega\setminus\left\{  x\right\}  ,$ as proved in \cite{LqManf}.

To prove that $u_{\infty}$ attains its maximum value $1$ uniquely at $x$ we
apply the comparison principle for the $\infty$-harmonic equation by using the
function
\[
v(y):=1-m^{-1}\left\vert y-x\right\vert ,\quad y\in\Omega
\]
where $m:=\max\left\{  \left\vert y-x\right\vert :y\in\partial\Omega\right\}
.$

In fact, as it is easy to check, $\Delta_{\infty}v=0$ in $\Omega
\setminus\left\{  x\right\}  $ and $u_{\infty}\leq v$ on $\partial
(\Omega\setminus\left\{  x\right\}  )=\left\{  x\right\}  \cup\partial\Omega.$
Therefore, since $\Delta_{\infty}u_{\infty}=0$ in $\Omega\setminus\left\{
x\right\}  $ the comparison principle yields
\[
u_{\infty}(y)\leq v(y)=1-m^{-1}\left\vert y-x\right\vert <1=\left\Vert
u_{\infty}\right\Vert _{\infty},\quad\forall\,y\in\Omega\setminus\left\{
x\right\}  .
\]

\end{proof}

\begin{remark}
When $\Omega$ is convex $u_{p}$ is nondecreasing with respect to $p$ in
$\Omega\setminus\left\{  x\right\}  $ (see \cite[Lemma 2.4]{Jan}): if
$N<p_{1}<p_{2}$ then $u_{p_{1}}(y)\leq u_{p_{2}}(y)$ for all $y\in
\Omega\setminus\left\{  x\right\}  $. Thus, in this case, the convergence of
$u_{p}\rightarrow u_{\infty}$ is also monotone.
\end{remark}

As for the unidimensional case, we observe from (\ref{un=1}) that
$u_{p}=u_{\infty}$. So, we can verify directly that $\Delta_{\infty}u_{\infty
}=0$ in $(a,x_{0})\cup(x_{0},b).$

\section{Concavity\label{Sec4}}

In this section we assume that $\Omega$ is convex and, based on the arguments
developed in Section 4 of \cite{HL}, we show that the function $s_{p}$ is
concave. The case $N=1$ follows from a simple analysis of the expression
\ref{spunid}. So, we consider $p>N\geq2$ in this section.

\begin{remark}
\label{prop2.8HL}As we are assuming that $\Omega$ is convex, for each
$x\in\Omega$ the punctured domain $\Omega\setminus\left\{  x\right\}  $ fits
in the definition of convex ring considered by Lewis in \cite{Lew}. As
mentioned in the Introduction, the positive minimizer $u_{p}$ of $\mu_{p}(x)$
on $\mathcal{M}_{p}(x),$ given by Proposition \ref{mumin}, is the
$p$-capacitary function of $\Omega\setminus\left\{  x\right\}  .$ Thus,
according to Theorem 1 of \cite{Lew}, $u_{p}$ is real analytic in
$\Omega\setminus\left\{  x\right\}  $ and $\left\vert \nabla u_{p}\right\vert
\not =0$ in this domain. Moreover (see Proposition 2.8 and Remark 2.9 of
\cite{HL})%
\[
\lim_{y\rightarrow x}\frac{\left\vert u_{p}(y)-1\right\vert }{\left\vert
y-x\right\vert ^{\frac{p-N}{p-1}}}=\frac{p-1}{p-N}\left(  \frac{\mu_{p}%
(x)}{N\omega_{N}}\right)  ^{\frac{1}{p-1}}\quad\mathrm{and}\quad
\lim_{y\rightarrow x}\left\vert \nabla u_{p}(y)\right\vert \left\vert
y-x\right\vert ^{\frac{N-1}{p-1}}=\left(  \frac{\mu_{p}(x)}{N\omega_{N}%
}\right)  ^{\frac{1}{p-1}}.
\]

\end{remark}

\begin{theorem}
\label{convex1}If $\Omega$ is a bounded, convex domain of $\mathbb{R}^{N}$,
then the function $s_{p}$ is concave in $\Omega.$
\end{theorem}

\begin{proof}
Since
\[
s_{p}=(\mu_{p})^{-1/p}=\left(  (\mu_{p})^{-1/(p-N)}\right)  ^{(p-N)/p}%
\quad\mathrm{and}\quad(p-N)/p\in(0,1)
\]
the concavity of $s_{p}$ follows once we prove that $(\mu_{p})^{-1/(p-N)}$ is concave.

Thus, in order to prove the concavity of $(\mu_{p})^{-1/(p-N)}$ we fix
$x_{0},x_{1}\in\Omega$ and $\rho\in(0,1)$ and define%
\[
x_{\rho}:=(1-\rho)x_{0}+\rho x_{1}.
\]

Let $u_{0}\in\mathcal{M}_{p}(x_{0}),$ $u_{1}\in\mathcal{M}_{p}(x_{1})$ and
$u_{\rho}\in\mathcal{M}_{p}(x_{\rho})$ denote the normalized, positive
minimizers of $\mu_{p}(x_{0}),$ $\mu_{p}(x_{1})$ and $\mu_{p}(x_{\rho}),$ respectively.

In the sequel we consider the $\rho$-Minkowski combination of $u_{0}$ and
$u_{1},$ defined by
\[
v_{\rho}(z):=\sup\left\{  \min\left\{  u_{0}(x),u_{1}(y)\right\}
:z=(1-\rho)x+\rho y,\quad x,y\in\overline{\Omega}\right\}  .
\]

It is known that $v_{\rho}\in W_{0}^{1,p}(\Omega).$ Actually, $v_{\rho}%
\in\mathcal{M}_{p}(x_{\rho})$ since $v_{\rho}(x_{\rho})=\left\Vert v_{\rho
}\right\Vert _{\infty}=1$ (which is easy to verify). Hence,%
\[
\mu_{p}(x_{\rho})\leq\left\Vert \nabla v_{\rho}\right\Vert _{p}^{p}.
\]

Following the first three steps of the proof of Lemma 4.2 of \cite{HL} we can
show that
\begin{equation}
\mu_{p}(x_{\rho})\leq\liminf_{r\rightarrow0^{+}}\int_{\partial B_{r}(x_{\rho
})}\left\vert \nabla v_{\rho}\right\vert ^{p-1}\mathrm{d}\sigma\label{gauss1}%
\end{equation}
and%
\begin{equation}
v_{\rho}\leq u_{\rho}\quad\mathrm{in}\,\Omega. \label{v<u}%
\end{equation}

Now, by adapting the remaining of the proof of Lemma 4.2 of \cite{HL} we prove
in the sequel that
\begin{equation}
\limsup_{x\rightarrow x_{\rho}}\left\vert \nabla v_{\rho}(x)\right\vert
^{p-1}\left\vert x-x_{\rho}\right\vert ^{N-1}\leq\left[  N\omega_{N}(\mu
_{p}(x_{\rho}))^{\frac{N-1}{p-N}}\left(  \frac{1-\rho}{(\mu_{p}(x_{0}%
))^{\frac{1}{p-N}}}+\frac{\rho}{(\mu_{p}(x_{1}))^{\frac{1}{p-N}}}\right)
^{p-1}\right]  ^{-1}. \label{step2}%
\end{equation}

Assuming this for a moment, noticing that%
\[
\liminf_{r\rightarrow0^{+}}\int_{\partial B_{r}(x_{\rho})}\left\vert \nabla
v_{\rho}\right\vert ^{p-1}\mathrm{d}\sigma\leq\limsup_{n\rightarrow\infty
}N\omega_{N}\left\vert \nabla v_{\rho}(z)\right\vert ^{p-1}\left\vert
z-x_{\rho}\right\vert ^{N-1},
\]
and taking (\ref{gauss1}) and (\ref{step2}) into account we arrive at%
\[
\mu_{p}(x_{\rho})\leq\left[  (\mu_{p}(x_{\rho}))^{\frac{N-1}{p-N}}\left(
\frac{1-\rho}{(\mu_{p}(x_{0}))^{\frac{1}{p-N}}}+\frac{\rho}{(\mu_{p}%
(x_{1}))^{\frac{1}{p-N}}}\right)  ^{p-1}\right]  ^{-1},
\]
which leads to
\[
\left(  \mu_{p}(x_{\rho})\right)  ^{-1}\geq(\mu_{p}(x_{\rho}))^{\frac
{N-1}{p-N}}\left(  \frac{1-\rho}{(\mu_{p}(x_{0}))^{\frac{1}{p-N}}}+\frac{\rho
}{(\mu_{p}(x_{1}))^{\frac{1}{p-N}}}\right)  ^{p-1},
\]
or, equivalently, to
\[
\left(  \mu_{p}(x_{\rho})\right)  ^{-\frac{1}{p-N}}\geq(1-\rho)(\mu_{p}%
(x_{0}))^{-\frac{1}{p-N}}+\rho(\mu_{p}(x_{1}))^{-\frac{1}{p-N}}.
\]
This shows that the function $x\mapsto(\mu_{p})^{-1/(p-N)}$ is concave.

To prove (\ref{step2}) let us take $z_{n}\rightarrow x_{\rho}$ with
$z_{n}\not =x_{\rho}$ for all $n\in\mathbb{N}$ large enough. Properties of the
function $v_{\rho}$ (see Proposition 4.1 of \cite{HL}), guarantee the
existence of sequences $\left(  x_{n}\right)  $ and $\left(  y_{n}\right)  $
such that%
\[
z_{n}=(1-\rho)x_{n}+\rho y_{n}%
\]%
\begin{equation}
v_{\rho}(z_{n})=u_{0}(x_{n})=u_{1}(y_{n}) \label{v=u0=u1}%
\end{equation}
and%
\[
\frac{1}{\left\vert \nabla v_{\rho}(z_{n})\right\vert }=\frac{1-\rho
}{\left\vert \nabla u_{0}(x_{n})\right\vert }+\frac{\rho}{\left\vert \nabla
u_{1}(y_{n})\right\vert }.
\]

Since $\Omega$ is bounded, we can assume that the sequences $\left(
x_{n}\right)  $ and $\left(  y_{n}\right)  $ are convergent, say
$x_{n}\rightarrow\overline{x}$ and $y_{n}\rightarrow\overline{y}.$ It follows
from (\ref{v=u0=u1}) that%
\[
1=v_{\rho}(x_{\rho})=u_{0}(\overline{x})=u_{1}(\overline{y})
\]
where the first equality comes from the fact that $z_{n}\rightarrow x_{\rho}.$
Noting that $1$ is the maximum value of both $u_{0}$ and $u_{1},$ assumed only
at $x_{0}$ and $x_{1},$ respectively, we conclude that $\overline{x}=x_{0}$
and $\overline{y}=x_{1}.$ Thus, $x_{n}\rightarrow x_{0}$ and $y_{n}\rightarrow
x_{1}.$

Combining (\ref{v<u}) with (\ref{v=u0=u1}) we have%
\[
\frac{\left\vert x_{n}-x_{0}\right\vert ^{\frac{p-N}{p-1}}}{1-u_{0}(x_{n}%
)}\frac{1-u_{\rho}(z_{n})}{\left\vert z_{n}-x_{\rho}\right\vert ^{\frac
{p-N}{p-1}}}\leq\frac{\left\vert x_{n}-x_{0}\right\vert ^{\frac{p-N}{p-1}}%
}{1-u_{0}(x_{n})}\frac{1-v_{\rho}(z_{n})}{\left\vert z_{n}-x_{\rho}\right\vert
^{\frac{p-N}{p-1}}}=\left(  \frac{\left\vert x_{n}-x_{0}\right\vert
}{\left\vert z_{n}-x_{\rho}\right\vert }\right)  ^{\frac{p-N}{p-1}}.
\]
Thus, it follows from Remark \ref{prop2.8HL} that
\begin{align*}
\left(  \left(  \mu_{p}(x_{0})\right)  ^{\frac{1}{p-1}}\right)  ^{-1}\left(
\left(  \mu_{p}(x_{\rho})\right)  ^{\frac{1}{p-1}}\right)   &  =\lim
_{n\rightarrow\infty}\frac{\left\vert x_{n}-x_{0}\right\vert ^{\frac{p-N}%
{p-1}}}{1-u_{0}(x_{n})}\lim_{n\rightarrow\infty}\frac{1-u_{\rho}(z_{n}%
)}{\left\vert z_{n}-x_{\rho}\right\vert ^{\frac{p-N}{p-1}}}\\
&  \leq\liminf_{n\rightarrow\infty}\left(  \frac{\left\vert x_{n}%
-x_{0}\right\vert }{\left\vert z_{n}-x_{\rho}\right\vert }\right)
^{\frac{p-N}{p-1}},
\end{align*}
so that%
\[
\left(  \frac{\mu_{p}(x_{\rho})}{\mu_{p}(x_{0})}\right)  ^{\frac{1}{p-N}}%
\leq\liminf_{n\rightarrow\infty}\frac{\left\vert x_{n}-x_{0}\right\vert
}{\left\vert z_{n}-x_{\rho}\right\vert }.
\]

Likewise, we obtain%
\[
\left(  \frac{\mu_{p}(x_{\rho})}{\mu_{p}(x_{1})}\right)  ^{\frac{1}{p-N}}%
\leq\liminf_{n\rightarrow\infty}\frac{\left\vert y_{n}-x_{1}\right\vert
}{\left\vert z_{n}-x_{\rho}\right\vert }.
\]

Hence, as
\begin{align*}
\frac{1}{\left\vert \nabla v_{\rho}(z_{n})\right\vert \left\vert z_{n}%
-x_{\rho}\right\vert ^{\frac{N-1}{p-1}}}  &  =\frac{1-\rho}{\left\vert \nabla
u_{0}(x_{n})\right\vert \left\vert z_{n}-x_{\rho}\right\vert ^{\frac{N-1}%
{p-1}}}+\frac{\rho}{\left\vert \nabla u_{1}(y_{n})\right\vert \left\vert
z_{n}-x_{\rho}\right\vert ^{\frac{N-1}{p-1}}}\\
&  =\left(  \frac{\left\vert x_{n}-x_{0}\right\vert }{\left\vert z_{n}%
-x_{\rho}\right\vert }\right)  ^{\frac{N-1}{p-1}}\frac{1-\rho}{\left\vert
\nabla u_{0}(x_{n})\right\vert \left\vert x_{n}-x_{0}\right\vert ^{\frac
{N-1}{p-1}}}\\
&  +\left(  \frac{\left\vert y_{n}-x_{1}\right\vert }{\left\vert z_{n}%
-x_{\rho}\right\vert }\right)  ^{\frac{N-1}{p-1}}\frac{\rho}{\left\vert \nabla
u_{1}(y_{n})\right\vert \left\vert y_{n}-x_{1}\right\vert ^{\frac{N-1}{p-1}}},
\end{align*}
Remark \ref{prop2.8HL} yields%
\begin{align*}
\frac{1}{\limsup\limits_{n\rightarrow\infty}\left\vert \nabla v_{\rho}%
(z_{n})\right\vert \left\vert z_{n}-x_{\rho}\right\vert ^{\frac{N-1}{p-1}}}
&  \geq\left(  \frac{\mu_{p}(x_{\rho})}{\mu_{p}(x_{0})}\right)  ^{\frac
{1}{p-N}\frac{N-1}{p-1}}\left(  \frac{N\omega_{N}}{\mu_{p}(x_{0})}\right)
^{\frac{1}{p-1}}(1-\rho)\\
&  +\left(  \frac{\mu_{p}(x_{\rho})}{\mu_{p}(x_{1})}\right)  ^{\frac{1}%
{p-N}\frac{N-1}{p-1}}\left(  \frac{N\omega_{N}}{\mu_{p}(x_{1})}\right)
^{\frac{1}{p-1}}\rho\\
&  =(N\omega_{N})^{\frac{1}{p-1}}(\mu_{p}(x_{\rho}))^{\frac{1}{p-N}\frac
{N-1}{p-1}}\left(  \frac{1-\rho}{(\mu_{p}(x_{0}))^{\frac{1}{p-N}}}+\frac{\rho
}{(\mu_{p}(x_{1}))^{\frac{1}{p-N}}}\right)  .
\end{align*}

Inequality (\ref{step2}) then follows from the arbitrariness of $z_{n}%
\rightarrow x_{\rho}.$
\end{proof}

\section{Acknowledgments}

The first author thanks the support of Funda\c{c}\~{a}o de Amparo \`{a}
Pesquisa do Estado de Minas Gerais - Fapemig/Brazil (CEX-APQ-03372-16 and
PPM-00137-18) and Conselho Nacional de Desenvolvimento Cient\'{\i}fico e
Tecnol\'{o}gico - CNPq/Brazil (306815/2017-6 and 422806/2018-8). The second
author thanks the support of Coordena\c{c}\~{a}o de Aperfei\c{c}oamento de
Pessoal de N\'{\i}vel Superior - Capes/Brazil (Finance Code 001).


\begin{thebibliography}{99}                                                                                               %


\bibitem {BB}Barles, G., Busca, J.: Existence and comparison results for fully
nonlinear degenerate elliptic equations without zeroth-order term, Comm. PDE
\textbf{26} (2001) 2323--2337.

\bibitem {DJM}Dinca, G., Jebelean, P., Mawhin, J.: Variational and topological
methods for Dirichlet problems with p-Laplacian, Port. Math. \textbf{58}
(2001) 339--378.

\bibitem {EP}Ercole, G., Pereira, G.: Asymptotics for the best Sobolev
constants and their extremal functions, Math. Nachr. \textbf{289} (2016) 1433--1449.

\bibitem {HL}Hynd, R., Lindgren, E.: Extremal functions for Morrey's
inequality in convex domains, Math. Ann. \textbf{375} (2019) 1721--1743.

\bibitem {Jan}Janfalk, U.: Behaviour in the limit, as $p\rightarrow\infty,$ of
minimizers of functionals involving p-Dirichlet integrals. SIAM J. Math.
Anal., \textbf{27} (1996) 341--360.

\bibitem {Jensen}Jensen, R.: Uniqueness of Lipschitz extensions: minimizing
the sup norm of the gradient, Arch. Rational Mech. Anal. \textbf{123} (1993) 51--74.

\bibitem {Lew}Lewis, J.: Capacitary functions in convex rings, Arch. Rational
Mech. Anal. \textbf{66} (1977) 201--224.

\bibitem {Lq1}Lindqvist, P.: Notes on the infinity Laplace equation.
SpringerBriefs in Mathematics. BCAM Basque Center for Applied Mathematics,
Springer, Bilbao, 2016.

\bibitem {Lq}Lindqvist, P.: Notes on the p-Laplace equation (2nd edition). No.
161. University of Jyv\"{a}skyl\"{a},Jyv\"{a}skyl\"{a}, 2017.

\bibitem {LqManf}Lindqvist, P., Manfredi, J.: The Harnack inequality for
$\infty$-harmonic functions, Electron. J. Differential Equations 1995 No. 4
(1995) 1--5.
\end{thebibliography}
\end{document}